\documentclass[12pt]{article}
\usepackage[utf8]{inputenc}
\usepackage{amsfonts}
\usepackage{scrextend}
\usepackage{mathtools}
\usepackage{float}
\usepackage{amsmath}
\usepackage[margin=1.2in]{geometry}
\usepackage{authblk}
\usepackage{amssymb}
\usepackage{epigraph}
\usepackage[english]{babel}
\usepackage[nottoc]{tocbibind}
\usepackage{tikz-cd}
\usepackage{graphicx}
\usetikzlibrary{matrix}
\usepackage{amsthm, enumitem}
\usepackage[mathscr]{euscript}
 \let\mathscr\relax
\usepackage[scr]{rsfso}
\usepackage{comment}

\usepackage{bm, mathdots, accents}

\theoremstyle{definition}
\newtheorem{nummer}{ }[section]

\newtheorem{thm}[nummer]{\sc Theorem}
\newtheorem*{mainthm}{\sc Main Theorem}
\newtheorem{prp}[nummer]{\sc Proposition}
\newtheorem{lem}[nummer]{\sc Lemma}

\newtheorem{defi}[nummer]{\sc Definition}
\newtheorem{rmk}{\sc Remark}
\newtheorem{problem}{\sc Problem}

\newcounter{faelle} 
\renewcommand{\thefaelle}{\rm(\alph{faelle})}

\renewcommand\qed{\relax\ifmmode~\hfill$\dashv$\else\unskip\nobreak~\hfill$\dashv$\fi}

\def\epsilon{\varepsilon}

\newcommand\func{{}^\omega\omega}

\newcommand\ran{\operatorname{ran}}

\newcommand\fin{\operatorname{fin}}

\newcommand\cov{\mathit{cov}}

\newcommand\forcing{\mathbb{M}(\bar{\mathcal{U}})}

\renewcommand{\phi}{\varphi}
\renewcommand{\theta}{\vartheta}

\begin{document}

\begin{center}
{\Large\sc There may be exactly $n$ $Q$-points}\\[1.8ex]

{\small Lorenz Halbeisen}\\[1.2ex] 
{\scriptsize Department of Mathematics, ETH Z\"urich, 
8092 Z\"urich, Switzerland\\ 
lorenz.halbeisen@math.ethz.ch}\\[1.8ex]

{\small Silvan Horvath}\\[1.2ex] 
{\scriptsize Department of Computer Science, ETH Z\"urich, 
8092 Z\"urich, Switzerland\\ 
silvan.horvath@inf.ethz.ch}\\[1.8ex]

{\small Tan {\"O}zalp}\\[1.2ex] 
{\scriptsize Department of Mathematics, University of Notre Dame,
Notre Dame, IN 46556, USA\\
aozalp@nd.edu}\\[1.8ex]

\end{center}


\begin{quote}
{\small {\bf Abstract.} 
We generalize the main result of \cite{halbeisen2025there} and show the consistency of the statement ``There are exactly $n$ $Q$-points up to isomorphism" for any finite $n$. Furthermore, we show that the above statement for $n=2$ can alternatively be obtained by a length-$\omega_2$ countable support iteration of Matet-Mathias forcing restricted to a Matet-adequate family.}
\end{quote}

\begin{quote}
\small{{\bf key-words\/}: $Q$-point,  Ramsey ultrafilter, Mathias forcing, iterated proper forcing, combinatorics with block sequences, Matet-adequate families}\\
\small{\bf 2010 Mathematics Subject Classification\/}: {\bf 03E35}\ {03E17}\ {03E05}\ {05C55}
\end{quote}


\setcounter{section}{0}
\section{Introduction}
 
 In a recent paper \cite{halbeisen2025there}, the first two authors and Shelah have shown that there is a unique $Q$-point in the model obtained by iterating Mathias forcing restricted to an increasing sequence of Ramsey ultrafilters $\omega_2$-many times. In this paper, we generalize this result and show that after iterating an $n$-dimensional variant of restricted Mathias forcing introduced by Shelah and Spinas~\cite{shelah1998distributivity}, the resulting model contains exactly $n$ $Q$-points. For an introduction to $Q$-points and Ramsey ultrafilters, see Blass' article in the handbook~\cite{blass2009combinatorial}.
 
Furthermore, we consider a variation of a recent model by Mildenberger~\cite{mildenberger2024exactly}, which contains exactly two $Q$-points as a consequence of containing exactly three \emph{near-coherence classes of ultrafilters} (defined below). We show that replacing Matet forcing restricted to a \emph{Matet-adequate family} with its Mathias-style variant, also yields a model with exactly two $Q$-points, while its number of near-coherence classes is $2^{\mathfrak{c}}$.

It is well-known that there are $2^{\mathfrak{c}}$ $Q$-points in models with $\cov(\mathcal{M})=\mathfrak{d}$ (see Mill{\'a}n~\cite{millan2007note}), and that there are no $Q$-points in the Mathias- and Laver models (see Miller~\cite{miller1980there}), as well as in models in which all ultrafilters are nearly coherent (see Blass~\cite{blass86nearcoherence}) -- such as the Miller model (see Blass-Shelah~\cite{blass1989near}).

\section{A Model with Exactly $n$ $Q$-points}

\begin{mainthm}
Let $n \in \omega$. It is consistent that there are exactly $n$ $Q$-points.
\end{mainthm}

Since models containing no $Q$-points are well-known, we consider the case $n > 0$. Our construction is a straightforward generalization of the method in \cite{halbeisen2025there}, which is why we abstain from an overly detailed presentation and focus on showing how the necessary lemmata from \cite{halbeisen2025there} translate to the general setting. 

\begin{defi}
Let $\bar{\mathcal{U}}=\langle\mathcal{U}_0, \mathcal{U}_1, ..., \mathcal{U}_{n-1} \rangle$ be a finite sequence of non-isomorphic Ramsey ultrafilters. Let $\forcing$ consist of conditions $\langle s, \bar X\rangle$, where
\begin{itemize}
\item $s \in \fin(\omega)$,
\item $\bar X \in \prod_{k \in n}\mathcal{U}_k \cap [\omega \setminus (\max s)^{+}]^{\omega}$.
\end{itemize}
If $\langle m_0, m_1, ..., m_{|s|-1}\rangle$ is the increasing enumeration of $s$, we denote by $s_k \in \fin(\omega)$ the set $s_k:=\{m_i: i \equiv k \mod n\}$. We define $\langle s, \bar X\rangle \leq_{\forcing} \langle t, \bar Y \rangle$ if and only if
\begin{itemize}
\item $s \supseteq t$
\item $\forall k \in n: \bar X(k) \subseteq \bar Y(k)$
\item $\forall k \in n: s_k \setminus t_k \subseteq \bar Y (k)$
\end{itemize}
We write $\bar X \in \bar{\mathcal{U}}$ if and only if $\bar X \in \prod_{k \in n}\mathcal{U}_k$. For $\bar X, \bar Y \in \bar{\mathcal{U}}$, we write $\bar X \leq_{\bar{\mathcal{U}}}\bar Y$ if and only if $\bar X(k) \subseteq \bar Y(k)$ for every $k \in n$. For $ m\in \omega$, we denote by $\bar X \setminus m$ the sequence $\langle \bar X(k)\setminus m: k \in n\rangle$. Finally, if $G$ is $\forcing$-generic, we denote by ${\eta}$ the generic real $\bigcup \{s: \exists \langle s, \bar X\rangle \in G\}$, and by $\eta_k$ the real $\{i \in \eta: i \equiv k \mod n\}$, which is a pseudo-intersection of $\mathcal{U}_k$.
\end{defi}

We begin by listing the basic properties of $\bar{\mathcal{U}}$ and $\forcing$ we will make use of. The first is a very useful game characterization proved by Shelah and Spinas. For $i \in \omega$, we denote by $i_{\text{mod}\, n}$ the element in $n$ with $i \equiv i_{\text{mod}\, n} \mod n$

\begin{defi}
Let $\bar{\mathcal{V}}=\langle \mathcal{V}_0, \mathcal{V}_1, ..., \mathcal{V}_{n-1}\rangle$ be a sequence of ultrafilters. The game $G(\bar{\mathcal{V}})$ is played between the Maiden and Death and proceeds as follows: In the $i$'th round, the Maiden plays some $\bar X_i \in \bar{\mathcal{V}}$ and Death responds by playing some $n_i \in \bar X_i(i_{\text{mod}\, n})$. Death wins if and only if for each $k \in n: \{n_i: i \equiv k \mod n\}\in \mathcal{V}_k$.
\end{defi}

\begin{lem}[{\cite[Corollary~1.3]{shelah1998distributivity}}]
If $\bar{\mathcal{U}}$ is a sequence of pairwise non-isomorphic Ramsey ultrafilters, the Maiden has no winning strategy in the game $G(\bar{\mathcal{U}})$.
\end{lem}

\begin{lem}[{\cite[Lemma~1.15]{shelah1998distributivity}}]
$\forcing$ has the pure decision property, i.e., for any sentence $\phi$ in the forcing language and any $\forcing$-condition $\langle s, \bar X \rangle$, there exists $\bar Y \leq_{\mathcal{U}} \bar X$ such that either $\langle s,\bar Y\rangle \Vdash_{\forcing} \phi$ or $\langle s,\bar Y\rangle \Vdash_{\forcing} \neg \phi$.
\end{lem}

\begin{lem}\label{lem:laver}
$\forcing$ has the Laver property, i.e., if $\undertilde{g}$ is a $\forcing$-name for an element of ${^\omega}\omega$ such that there exists $f\in {^\omega}\omega \cap \mathbf{V}$ with $\forcing\Vdash \forall i \in \omega: \undertilde{g}(i) \leq f(i)$,
then $\forcing$ forces that there is some $c: \omega \to \text{fin}(\omega)$ in $\mathbf{V}$ with
\[\forall i \in \omega: |c(i)| \leq 2^{i} \text{ and } \undertilde{g}(i) \in c(i).\]
\end{lem}

\begin{proof}
Let $\undertilde{g}$ and $f$ be as in the statement of the lemma and let $\langle s, \bar X\rangle \in \forcing$. We play the game $G(\bar{\mathcal{U}})$.

By the pure decision property, there exists $\bar Y_0 \leq_{\bar{\mathcal{U}}}\bar X$ and $k^0_{\emptyset} \leq f(0)$ such that $\langle s, \bar Y_0\rangle \Vdash_{\forcing} \undertilde{g}(0)=k^0_{\emptyset}$. The Maiden starts by playing $\bar Y_0$.

Assume $\bar Y_0 \geq_{\bar{\mathcal{U}}} \bar Y_1 \geq_{\bar{\mathcal{U}}} ... \geq_{\bar{\mathcal{U}}} \bar Y_i$ and $n_0 < n_1 < ... < n_i$ have been played, where $\forall j \leq i: n_j \in \bar Y_j(j_{\text{mod}\, n})$. For each $t \subseteq \{n_j: j \leq i\}$, there is some $k^{i+1}_t \leq f(i+1)$ and some $\bar Z_t \leq_{\bar{\mathcal{U}}}\bar Y_i \setminus n_i^{+}$ such that $\langle s \cup t, \bar Z_t\rangle \Vdash_{\forcing} \undertilde{g}(i+1)=k_t^{i+1}$. Let $\bar Y_{i+1}$ be the component-wise intersection of these $\bar Z_t$.

Since this is not a winning strategy, Death wins and the set $\bar Y(k):=\{n_i: i \equiv k \mod n\}$ is in $\mathcal{U}_k$ for every $k \in n$. It is easy to check that
\[\langle s, \bar Y\rangle \Vdash_{\forcing} \forall i \in \omega: \undertilde{g}(i) \in \{k_t^i: t \subseteq \{n_j: j < i\}\}.\]

\vspace{-1cm}
\qedhere
\end{proof}

Now, the construction of our model is as follows: We work over a ground model satisfying $\textsf{CH}$ and define a countable support iteration $\mathbb{P}_{\omega_2}:=\langle \mathbb{P}_{\xi},, \mathbb{M}(\bar{\undertilde{\mathcal{U}}}_{\xi}): \xi \in \omega_2\rangle$, where for each $\xi \in \omega_2$: 
\[ \begin{split} \mathbb{P}_\xi \Vdash &``\;\undertilde{\bar{\mathcal{U}}}_{\xi}=\langle \undertilde{\mathcal{U}}_0^{\xi}, \undertilde{\mathcal{U}}_1 ^{\xi}, ..., \undertilde{\mathcal{U}}_{n-1} ^{\xi}\rangle \text{ are non-isomorphic Ramsey ultrafilters} \\ &\text{such that } \forall \iota < \xi \;\forall k \in n: \undertilde{\mathcal{U}}_k^{\iota} \subseteq \undertilde{\mathcal{U}}_k^{\xi}. \;" \end{split}\]
At each successor stage $\xi+1 \in \omega_2$, we construct $\undertilde{\mathcal{U}}_k^{\xi+1}$ as an extension of the pseudo-intersection $\underaccent{\dot}{\eta}_k^{\xi}$ of $\undertilde{\mathcal{U}}_k^{\xi} $. This is possible by the \textsf{CH} at all intermediate stages. At limit stages $\xi$ of cofinality $\omega$, we construct $\undertilde{\mathcal{U}}_k^{\xi+1}$ on a pseudo-intersection of the tower $\langle \eta_k^{\iota}: \iota \in \xi\rangle$, and at limit stages of uncountable cofinality, $\bigcup_{\iota \in \xi}\undertilde{\mathcal{U}}_k^{\iota}$ is already forced to be a Ramsey ultrafilter, since no new reals appear at such stages.

Note that $\mathbb{P}_{\omega_2}$ is proper and satisfies the $\omega_2$-c.c. (e.g., see {~\cite[Theorem~2.10]{abraham2009proper}}). We show that for every $\mathbb{P}_{\omega_2}$-generic $G$, the Ramsey ultrafilters
\[\mathcal{U}_k^{\omega_2}:=\bigcup_{\xi \in \omega_2} \undertilde{\mathcal{U}}_k^{\xi} \,[G|_{\xi}]\]
for $k \in n$ are the only $Q$-points in $\mathbf{V}[G]$. To see this, assume by contradiction that $\mathbf{V}[G]\models ``E\text{ is a }Q\text{-point and not isomorphic to any of the }\mathcal{U}_k^{\omega_2}"$. By a standard closure argument, there exists $\delta \in \omega_2$ such that $E \cap \mathbf{V}[G|_\delta] \in \mathbf{V}[G|_\delta]$ and $\mathbf{V}[G|_\delta] \models ``E \cap \mathbf{V}[G|_\delta] \text{ is a }Q\text{-point and not isomorphic to any of the }\mathcal{U}_k^\delta"$ (see \cite{halbeisen2025there}).

Hence, by working in $\mathbf{V}[G|_\delta]$ and factoring the tail of the iteration as $\forcing \ast \undertilde{R}$, where $\undertilde{R}$ is forced to have the Laver property, the main theorem follows directly from the following proposition.

\begin{prp}\label{prop:destruction}
	Let $E$ be a $Q$-point and $\bar{\mathcal{U}}=\langle \mathcal{U}_k: k \in n\rangle$ a finite sequence of $n$ pairwise non-isomorphic Ramsey ultrafilters such that $E$ is not isomorphic to any of the $\mathcal{U}_k$. Let $\undertilde{R}$ be a $\forcing$-name such that $\forcing \Vdash ``\undertilde{R}$ has the Laver property". Then $\forcing \ast \undertilde{R}\Vdash ``E$ cannot be extended to a $Q$-point".
\end{prp}

Proposition~\ref{prop:destruction} is proven through two lemmata. The first is analogous to Lemma~2.4 and Lemma~2.5 in \cite{halbeisen2025there}.

\begin{lem}\label{lem:deciding}
	Let $\langle s, \bar X\rangle \in \forcing$ and let $\undertilde{C}$ be an $\forcing$-name for a subset of $\omega$. There exists $\langle s, \bar X^{\ast}\rangle \leq_{\forcing} \langle s, \bar X\rangle$ such that, for every $t \in \fin(\bigcup_{k \in n}\bar X^{\ast}(k))$, there is some $C_t \in \fin(\omega)$ with
	\[\langle s \cup t, \bar X^{\ast} \setminus (\max t)^{+}\rangle \Vdash_{\forcing} \undertilde{C}\cap (\max t)^{+}=C_t.\]
	Furthermore, for any $m \in \bigcup_{k \in n} \bar X^{\ast}(k) \setminus (\max t)^{+}$ and all $x_{\varepsilon} > m$, $\epsilon \in \{0,1\}$, such that $\langle s \cup t \cup \{x_{\varepsilon}\}, \bar X^{\ast} \setminus x_{\varepsilon}\rangle \leq_{\forcing} \langle s \cup t, \bar X^{\ast}\rangle$, we have $C_{t \cup \{x_0\}} \cap m^{+}= C_{t \cup \{x_1\}} \cap m^{+}$.
\end{lem}

\begin{proof}
We find an $\bar X' \leq_{\bar{\mathcal{U}}}\bar X$ that satisfies the first part of the lemma as in the proof of Lemma~\ref{lem:laver}. To make the second part of the lemma work, we again play the game $G(\bar{\mathcal{U}})$. The Maiden starts by playing $\bar Y_0':= \bar X'$. Assume that $\bar Y_0' \geq_{\bar{\mathcal{U}}} \bar Y_1' \geq_{\bar{\mathcal{U}}} ... \geq_{\bar{\mathcal{U}}} \bar Y_i'$ and $n_0' < n_1' < ... < n_i'$ have been played. For each $t \subseteq \{n_j: j \leq i\}$, let $l_t := |s \cup t|_{\text{mod}\;n}$ and define for each $d \subseteq n_i^{+}$ the set
\[P_{t,d}:=\{x \in \bar Y_i'(l_t) \setminus n_i^{+}: C_{t \cup \{x\}} \cap n_i^{+}=d\}.\]

Note that for every $t \subseteq \{n_0, n_1, ..., n_i\}$, the set $\{P_{t,d}: d \subseteq n_i^{+}\}$ is a partition of $\bar Y_i'(l_t) \setminus n_i^{+}$ into finitely many pieces. Hence, there exists one $d_t\subseteq n_i^{+}$ such that $P_{t,d_t}\in \mathcal{U}_{l_t}$. Let
\[\bar Y'_{i+1}(k):= \bigcap_{\substack{{t \subseteq \{n_j: j \leq i\}}\\{l_t=k}}} P_{t,d_t}.\]
If there is no $t \subseteq \{n_j: j \leq i\}$ with $l_t=k$, let $Y_{i+1}'(k):= \bar Y_{i}'(k)$.

Again Death wins and the sequence $\bar X^{\ast}$ with $\bar X^{\ast}(k):=\{n_i: i \equiv k \mod n\}$ works.
\end{proof}

The preceding lemma implies the following analog of Lemma~2.7 in \cite{halbeisen2025there}.

\begin{lem}\label{lem:close}
	Let $\undertilde{C}$ be a $\forcing$-name for a subset of $\omega$ and let $\langle s, \bar X\rangle \in \forcing$ be such that $\langle s, \bar X\rangle \Vdash_{\forcing} \forall i \in \omega: |\undertilde{C} \cap {\underaccent{\dot}{\eta}(i)}^{+}|\leq 2^{i}$.
	
There exists an interval partition $\{[k_i, k_{i+1}): i \in \omega\}$ of $\omega$ and a strengthening of $\langle s, \bar X\rangle$ which forces that
		\[\forall i \in \omega\setminus \{0,1\}: \undertilde{C} \setminus (\max s)^{+}\cap [k_{i-1}, k_{i})\neq \emptyset \implies \exists j \in n:
	\begin{cases}
	\text{range}(\underaccent{\dot}{\eta}_j)\cap [k_{i-2}, k_{i-1})\neq \emptyset, \text{ or}	\\
	\text{range}(\underaccent{\dot}{\eta}_j)\cap [k_{i-1}, k_i)\neq \emptyset, \text{ or}	\\
	\text{range}(\underaccent{\dot}{\eta}_j)\cap [k_{i}, k_{i+1})\neq \emptyset.
	\end{cases}
\]
\end{lem}

\begin{proof}
Up to one minor modification, the proof is the same as the proof of Lemma~2.7 in \cite{halbeisen2025there}, which is why we will be brief. We first thin out $\bar X$ and extend $s$ by one element such that the $j$'th element of $\{\max s\} \cup \bigcup_{j \in n} \bar X(j)$ in increasing order is greater than $2^{j+1}$. Then, we apply Lemma~\ref{lem:deciding}, strengthening  $\langle s, \bar X\rangle$ to $\langle s, \bar X^{\ast}\rangle$.

To obtain the interval partition, we let $N$ be a countable elementary submodel of some large enough $\mathcal{H}_{\chi}$ such that $N$ contains all the relevant parameters. By induction, we construct a sequence $N_0 \subseteq N_1 \subseteq ... $ of finite subsets of $N$ with $\bigcup_{i \in \omega}N_i =N$, such that all the relevant parameters are in $N_0$, such that the $N_i$ are sufficiently closed, with $\forall i \in \omega: k_i:=N_i \cap \omega \in \omega$, and such that every formula $\phi(x, a_0, ...,a_l)$ of length less than $8400$ with parameters $a_j \in N_i$ and $N \models \exists x \phi(x, a_0, ...,a_l)$ has a witness in $N_{i+1}$.

To show that the interval partition given by the $k_i$ works, assume that $\langle s \cup t, \bar X'\rangle\geq_{\forcing} \langle s, \bar X^{\ast}\rangle$ and $a \in \omega \setminus (\max s)^{+}$ are such that $\langle s \cup t, \bar X'\rangle \Vdash_{\forcing} a \in \undertilde{C}$. By Lemma~\ref{lem:deciding} and by possibly extending $t$, we may assume without loss of generality that $\bar X'=\bar X^{\ast} \setminus (l^{\ast})^{+}$ and $l_0 < a \leq l^{\ast}$, where $l_0$ and $l^{\ast}$ are the second largest and the largest element of $t$, respectively.

If there is no element of $\bigcup_{j \in n}\bar X^{\ast}(j)$ between $a$ and $l^{\ast}$, we are done, since then $l^{\ast}$ is definable from $a$ by the formula $l^{\ast}=\min(\bigcup_{j \in n} \bar X^{\ast}(j) \setminus a)$. Hence we may assume that there is such an element $m^{\ast}$, i.e., $l_0 < a < m^{*} < l^{*}$.

We let $t':=t \setminus \{l^{\ast}\}$. It remains to show that if $t'$ and hence $l_0$ are in $N_{i^{\ast}}$, then some number greater than $a$ is in $N_{i^{\ast}+1}$. 

We let $M \in \omega$ be such that $l^{*}$ is the $M$'th element of $s \cup t$ in increasing order and define $n^{\ast}:=M_{\text{mod}\;n}$, i.e., $l^{\ast} \in \bar X^{\ast}(n^{\ast})$. Furthermore, we set $D^{*}:= C_{t' \cup \{l^{*}\}} \cap (l_0, m^{*})$. Note that $a \in D^{*}$ and that $|D^{*}|=:\gamma \leq 2^{M}$. By the initial thinning of $\bar X$, the parameter $\gamma$ is in $N_{i^{\ast}}$, and hence $N_{i^{\ast}+1}$ contains a witness $\langle m^{\dagger}, l^{\dagger}, D^{\dagger}\rangle$ to the formula

\[\exists \langle m,l, D\rangle: \begin{cases}
 	 m \in \bigcup_{j \in n}\bar X(j) \setminus {l_0}^{+},\; m<l, \;l \in \bar X(n^{\ast}) \setminus {l_0}^{+}\text{ and} \\
 	  D \subseteq (l_0, m), \text{ and } \\
 	  |D|=\gamma, \text{ and } \\
 	  \langle s \cup t' \cup \{l\}, \bar X \setminus l^{+}\rangle \Vdash_{\forcing} \undertilde{C} \cap (l_0, m)=D.
 \end{cases}
 \]

Now, $l^{\dagger}$ is the required number greater than $a$: If we had $m^{\dagger}< l^{\dagger} < a < m^{\ast}$, the two conditions $\langle s \cup t' \cup \{l^{\ast}\}, \bar X^{\ast} \setminus (l^{\ast})^{+}\rangle$ and $\langle s \cup t' \cup \{l^{\dagger}\}, \bar X^{\ast} \setminus (l^{\dagger})^{+}\rangle$ would agree on $\undertilde{C} \cap (l_0, m^{\dagger})$, by Lemma~\ref{lem:deciding}. Hence we would have $D^{\dagger} \subseteq D^{\ast}$, which in turn implies $D^{\dagger} = D^{\ast}$ because both have size $\gamma$. This would contradict $a \in D^{\ast}$.
\end{proof}

\begin{proof}[Proof of Proposition~\ref{prop:destruction}]
Assume that $E$ be a $Q$-point and $\bar{\mathcal{U}}=\langle \mathcal{U}_k: k \in n\rangle$ a sequence of $n$ pairwise non-isomorphic Ramsey ultrafilters such that $E$ is not isomorphic to any of the $\mathcal{U}_k$. Let $\undertilde{a}$ be a $\forcing \ast \undertilde{R}$-name for a strictly increasing element of ${^\omega}\omega$ such that
\[\forcing \ast \undertilde{R} \Vdash \forall i \in \omega: \undertilde{a}(i) \in (\underaccent{\dot}{\eta}(i-1), \underaccent{\dot}{\eta}(i)].\]

Let $\langle p, \undertilde{q}\rangle \in \forcing \ast \undertilde{R}$. By the Laver property of $\undertilde{R}$, there exists a $\forcing$-name $\undertilde{C}$ for a subset of $\omega$ and some $\langle p', \undertilde{q}'\rangle \leq_{\forcing \ast \undertilde{R}} \langle p, \undertilde{q}\rangle $ such that
\[\langle p', \undertilde{q}'\rangle \Vdash_{\forcing \ast \undertilde{R}} \forall i \in \omega: \undertilde{a}(i) \in \undertilde{C}\cap (\underaccent{\dot}{\eta}(i-1), \underaccent{\dot}{\eta}(i)] \text{ and } | \undertilde{C}\cap {\underaccent{\dot}{\eta}(i)}^{+} |\leq 2^i.\]

By Lemma~\ref{lem:close}, we find some $p^{\ast} \leq_{\forcing}p'$ and an interval partition $\{[k_i, k_{i+1}): i \in \omega\}$ of $\omega$ such that $p^{\ast}$ forces that eventually, $\undertilde{C}$ intersects some interval only if one of the reals $\eta_{k}$ for $k \in n$ intersects an adjacent interval. Hence, $\langle p^{\ast}, \undertilde{q}'\rangle$ forces the same for $\text{range}(\undertilde{a})$ in place of $\undertilde{C}$.

However, since $E$ and the $\mathcal{U}_k$ are non-isomorphic, it follows by a standard argument (e.g., see \cite[Lemma~2.8]{halbeisen2025there}) that there is $\bar Y \in \bar{\mathcal{U}}$ and $v \in E$ such that the $Y(k)$ and $v$ do not intersect adjacent intervals of the partition $\{[k_i, k_{i+1}): i \in \omega\}$. Since each $\eta_k$ is almost contained in $Y(k)$, $\undertilde{a}$ is thus forced to be almost disjoint from $v$.
\end{proof}

\section{Exactly Two $Q$-Points via Matet-Mathias Forcing}\label{section3}

In this section, we present an alternative proof of the main theorem for the case $n=2$, using machinery related to Milliken's topological Ramsey space $(\mathrm{FIN}^{[\infty]}, \leq)$ (see \cite{milliken}).

Two ultrafilters $\mathcal{U}_0$ and $\mathcal{U}_1$ are called \textit{nearly coherent} if and only if there is some finite-to-one $f \in \func$ such that $f(\mathcal{U}_0)=f(\mathcal{U}_1)$, where $f(\mathcal{U}_i):=\{x \subseteq \omega: f^{-1}[x]\in \mathcal{U}_i\}$. Near-coherence is an equivalence relation (see \cite{blass86nearcoherence}), and its equivalence classes are called \textit{near-coherence classes}. Note that two $Q$-points are nearly-coherent if and only if they are isomorphic. Banakh and Blass~\cite{banakh2006number} have shown that the number of near-coherence classes is either finite or $2^{\mathfrak{c}}$, and that the latter holds in models with $\mathfrak{u} \geq \mathfrak{d}$. Mildenberger~\cite{mildenberger2024exactly} recently constructed models with exactly two and with exactly three near-coherence classes.

It is easy to see that Mildenberger's models contain exactly one and exactly two $Q$-points, respectively: They contain at least that many Ramsey ultrafilters by construction and cannot contain more, since in both models, one of the near-coherence classes must contain a $\mathfrak{u}< \mathfrak{d}$-generated ultrafilter, and a $Q$-point cannot be ${<}\mathfrak{d}$-generated.

Mildenberger's model with exactly three near-coherence classes is obtained via a length-$\omega_2$ countable support iteration of Matet forcing restricted to a suitable \emph{Matet-adequate family} $\mathcal{H}$. In this section, we show that replacing Matet forcing in Mildenberger's construction with its Mathias-style variant, \emph{Matet-Mathias forcing} $\mathbb{MM}(\mathcal{H})$, the resulting model contains $2^{\mathfrak{c}}$ near-coherence classes instead of three, but its number of $Q$-points remains two.

\begin{thm}\label{thm:matetmathias}
Work over a ground model $\mathbf{V}$ satisfying CH. Then there is a countable support iteration of proper iterands $\mathbb{P}_{\omega_2}=\langle \mathbb{P}_{\alpha}, \undertilde{\mathbb{MM}}(\undertilde{\mathcal{H}}_{\alpha}) : \alpha<\omega_2\rangle$, such that there are exactly two $Q$-points in the extension by $\mathbb{P}_{\omega_2}$. 
\end{thm}

\begin{defi}

Let $\mathrm{FIN}=[\omega]^{<\omega}\setminus \{\varnothing\}$. For $s,t\in\mathrm{FIN}$, we write $s<_b t$ if $\max(s)<\min(t)$. By a \emph{block sequence}, we mean a $<_b$-strictly increasing sequence $X : \operatorname{dom}(X)$ $ \to$ $\mathrm{FIN}$, where $\operatorname{dom}(X)\in\omega\cup\{\omega\}$. The initial segment relation will be denoted by $\sqsubseteq$ and $\sqsubset$ will mean proper initial segment. The set of all finite block sequences of length $n \in \omega$ is denoted by $\mathrm{FIN}^{[n]}$, ${\mathrm{FIN}}^{[\infty]}$ denotes the set of all infinite block sequences and $\mathrm{FIN}^{[< n]} := \bigcup_{k < n} \mathrm{FIN}^{[k]}$, including the empty sequence.

In the remainder of the paper, capital letters $ X, Y, Z$ will denote elements of $\mathrm{FIN}^{[\infty]}$, small letters $a, b, c$ will denote elements of $\mathrm{FIN}^{[< \infty]}$ and $s, t, u$ will denote elements of $\mathrm{FIN}$.
\begin{enumerate}[label=(\roman*)]
    \item We write $s <_b a$ if and only if $s <_b a(0)$, and $a <_b b$ if and only if $a(|a|-1) <_b b(0)$, whenever $a,b \neq \varnothing$. Similarly, define $s <_b X$ and $a <_b X$.
    
    \item Assume that $a <_b b$. Then $a^{\smallfrown}b$ denotes the concatenation of $a$ and $b$.

    \item $[X] := \{s \in \mathrm{FIN} : s = \bigcup_{i \in I} X(i) \ \text{for some finite $I \subseteq \omega$}\}$, and $[a]$ is defined analogously. If for all $i \in \omega$, $X(i) \in [Y]$, we write $X \leq Y$ and call $X$ a \emph{condensation} of $Y$. The relations $a \leq X$ and $a \leq b$ are defined analogously. We denote by $[X]^{[\infty]}$ the set of infinite block sequences $X'$ with $X' \leq X$. For $n \in \omega$, $[X]^{[n]}$ denotes the set of block sequences $a \in \mathrm{FIN}^{[n]}$ with $a \leq X$, and $[X]^{[<n]} := \bigcup_{n' < n} [X]^{[n']}$, for $n\in\omega\cup\{\omega\}$.
    
    \item If $a\leq X$, then we define $d_X(a)$ to be the least integer $n\in\omega$ such that $a\leq X\upharpoonright (n+1)$. 

    \item Let $n \in \omega$. $X / n = \langle X(i_0+i): i \in \omega \rangle$, where $i_0$ is the least index with $\operatorname{min}(X(i_0)) > n$. For $s \in \mathrm{FIN}$, we will also write $X / s$ to mean $X / \operatorname{max}(s)$, and $X / a$ to mean $X / \operatorname{max}(a(|a|-1))$, for $\varnothing \neq a \in \mathrm{FIN}^{[<\infty]}$.  Define $a/n,\; a/s,\; a/b$ analogously.

    \item $X \leq^* Y$ if and only if there is some $n \in \omega$ with $X / n \leq Y$.
\end{enumerate}
\end{defi}

The famous theorem of Hindman and its later generalization of Taylor given below are the primary reasons that make this space interesting.
\begin{thm}[Hindman \cite{hind}]

If $c:\mathrm{FIN}\to 2$ is a coloring, then there is some $X\in\mathrm{FIN}^{[\infty]}$ such that $c\upharpoonright [X]$ is constant.
    
\end{thm}

\begin{thm}[Taylor \cite{Tay}]

Let $n,r$ be positive integers and $c:\mathrm{FIN}^{[n]}\to r$ some coloring. Then there is some $X\in\mathrm{FIN}^{[\infty]}$ such that $c\upharpoonright [X]^{[n]}$ is constant.

\end{thm}

\begin{defi}
Let $\mathcal{H}\subseteq \mathrm{FIN}^{[\infty]}$ be a nonempty family. We define
\[\mathbb{MM}(\mathcal{H})=\{\langle a, X\rangle : a\in\mathrm{FIN}^{[<\infty]}, a<_b X \in \mathcal{H}\}.\]
The order is given by $\langle a, X\rangle \leq_{\mathbb{MM}(\mathcal{H})} \langle b, Y\rangle$ if and only if $a\sqsupseteq b$, $a/b\leq Y$, and $X\leq Y$. We write $\mathbb{MM}$ for the unrestricted variant $\mathbb{MM}(\mathrm{FIN}^{[\infty]})$.
\end{defi}

Note the similarity of this forcing notion to traditional Matet forcing, which consists of conditions as above but with stems in $\mathrm{FIN}$ and with the relation $\sqsupseteq$ replaced by $\supseteq$. The (unrestricted) forcing notion $\mathbb{MM}$ was introduced by  Garc\'ia-\'Avila in \cite{PFIN}, and studied, for example, in \cite{BrendleGarciaAvila}, \cite{cov(M)<c}, and \cite{CalderonDiPriscoMijares}.

Garc\'ia-\'Avila showed that $\mathbb{MM}$ is proper and has the pure decision and Laver properties \cite{PFIN}. We will be concerned with $\mathbb{MM}(\mathcal{H})$ for certain families $\mathcal{H}\subseteq \mathrm{FIN}^{[\infty]}$, and to obtain properness, pure decision etc. for $\mathbb{MM}(\mathcal{H})$, we need $\mathcal{H}$ to satisfy certain conditions. In \cite{Eisworth}, Eisworth isolated the following properties of a family $\mathcal{H}$ that will suffice for this purpose. 

\begin{defi}[Eisworth \cite{Eisworth}]
For $\mathcal{H}\subseteq \mathrm{FIN}^{[\infty]}$ and $X\in\mathcal{H}$, define $\mathcal{H}|X:=\{Y\in\mathcal{H} : Y\leq X\}$. A nonempty family $\mathcal{H}\subseteq \mathrm{FIN}^{[\infty]}$ is called \emph{Matet-adequate} if and only if the following properties hold:

\begin{enumerate}[label=(\roman*)]
    \item $\mathcal{H}$ is closed under finite changes,
    \item $\mathcal{H}$ is closed $\leq^*$-upwards,
    \item For any $\leq$-descending sequence $X_0\geq X_1\geq \ldots$ of members of $\mathcal{H}$, there is some $X\in\mathcal{H}$ such that $X\leq^* X_n$ for all $n \in \omega$,
    \item For any $X\in \mathcal{H}$ and $c : [X]\to 2$, there is $Y\in\mathcal{H}|X$ such that $c \upharpoonright [Y]$ is constant.
\end{enumerate}

\end{defi}

As shown in \cite{Eisworth} and \cite{mildenberger2024exactly}, Matet-adequate families satisfy the following seemingly stronger properties as well:

\begin{lem}[Eisworth \cite{Eisworth}, Mildenberger \cite{mildenberger2024exactly}]\label{Matetadequateproperties}

Let $\mathcal{H}$ be Matet-adequate. Then the following properties also hold:

\begin{enumerate}[label=(\roman*)]
    \item For any positive integer $n$, $X\in\mathcal{H}$, and $c:[X]^n\to 2$, there is $Y\in\mathcal{H}|X$ such that $c\upharpoonright [Y]^n$ is constant.
    \item For any $X_0\geq X_1\geq\ldots$ from $\mathcal{H}$, there is $X\in\mathcal{H}|X_0$ such that for all $s\in[X]: X/s\leq X_{\max(s)+1}$,  (this is called a diagonal intersection of the sequence $\langle X_0, X_1,\ldots\rangle$).
\end{enumerate}
\end{lem}

Note that when $\mathcal{H}$ is Matet-adequate, a generic $G$ for $\mathbb{MM}(\mathcal{H})$ induces a generic infinite block sequence $\eta$. In \cite{CalderonDiPriscoMijares}, Calderon, Di Prisco, and Mijares prove that Matet-adequate families, and what they define as \emph{selective coideals} on $\mathrm{FIN}^{[\infty]}$, coincide (Theorem~5, \cite{CalderonDiPriscoMijares}). Consequently, we have the pure decision property relative to Matet-adequate families:

\begin{thm}[Theorem 10, Calderon-Di Prisco-Mijares \cite{CalderonDiPriscoMijares}]
   Assume that $\mathcal{H}\subseteq\mathrm{FIN}^{[\infty]}$ is a Matet-adequate family. Then for any condition $\langle a, X\rangle \in \mathbb{MM}(\mathcal{H})$, and any sentence $\varphi$ of the forcing language, there is $Y\in \mathcal{H}|X$ such that $\langle a, Y \rangle$ decides $\varphi$.
\end{thm}

Using this, it is straightforward to check that $\mathbb{MM}(\mathcal{H})$ is proper and has the $h$-Laver property for a suitably chosen function $h:\omega\to\omega$, whenever $\mathcal{H}$ is Matet-adequate. For completeness, we show the Laver property. Define $X_{\text{top}}(n)=\{n\}$ for all $n\in\omega$. Define $h:\omega\to\omega$ by $h(n)=|\{a\in\mathrm{FIN}^{[<\infty]}:a\leq X_{\text{top}}\upharpoonright (n+1)\}|$.

\begin{lem}\label{lemma:Laver property}

Assume that $\mathcal{H}$ is Matet-adequate. Then $\mathbb{MM}(\mathcal{H})$ has the $h$-Laver property: Whenever $\langle a, X\rangle\in\mathbb{MM}(\mathcal{H}), f\in\omega^\omega$, and $\undertilde{g}$ is a $\mathbb{MM}(\mathcal{H})$-name such that $\langle a, X\rangle\Vdash (\forall n) (\undertilde{g}(n)<f(n))$, there is some $Y\in\mathcal{H}|X$, and a function $H:\omega\to [\omega]^{<\omega}$ such that:
\begin{enumerate}[label=(\roman*)]
    \item $|H(n)|\leq h(n)$, for all $n\in\omega$,
    \item $\langle a, Y\rangle\Vdash (\forall n) \ \undertilde{g}(n)\in H(n)$.
\end{enumerate}
\end{lem}

\begin{proof}

Let $\langle a, X\rangle\in\mathbb{MM}(\mathcal{H})$, $f\in\omega^\omega$, $\undertilde{g}$ be as in the statement. By pure decision and since $\undertilde{g}$ is forced to be bounded by $f$, we find $X_0\in\mathcal{H}|X$ such that $\langle a, X_0\rangle$ decides $\undertilde{g}(0)$, say with value $k^0_\varnothing$. Set $H(0)=\{k^0_\varnothing\}$ and note that $|H(0)|\leq h(0)$. Starting with $X_0$, we will construct a descending chain $X_{n+1}\leq X_n/X(n)$ of elements of $\mathcal{H}|X_0$, and define $H(n)$ obeying the cardinality condition for each $n\in\omega$ such that, whenever $b\in[X]^{[<\infty]}$ and $d_X(b)=n+1$, $\langle a^{\smallfrown}b, X_{n+1}\rangle \Vdash \undertilde{g}(n+1)\in H(n+1)$. Assuming that this is accomplished, by Lemma~\ref{Matetadequateproperties}, we can find $Y\in\mathcal{H}|X_0$ which is a diagonal intersection of the sequence $\langle X_0, X_1, \ldots\rangle$, and it is easy to see that $\langle a, Y\rangle$ and $H$ are as required by the lemma.

For the construction itself, suppose $X_0,X_1,\ldots, X_n$ and $H\upharpoonright (n+1)$ are already constructed. List $\{ b_i : i \in m\}= [X\upharpoonright (n+1)]^{[<\infty]}$ for some $m\in\omega$, and note $m\leq h(n+1)$. Let $Z_0=X_n/X(n)$, and given $Z_i$ for $i<m$, choose $Z_{i+1}\in \mathcal{H}|Z_i$ and $k^{n+1}_{b_i}$ such that $\langle a^{\smallfrown} b_i, Z_{i+1}\rangle \Vdash \undertilde{g}(n+1)=k^{n+1}_{b_i}$. In the end, we set $X_{n+1}=Z_m$ and $H(n+1)=\{k^{n+1}_{b_i} : i\in m\}$.
\end{proof}

For $X\in\mathrm{FIN}^{[\infty]}$, write $\min[X]=\{\min(X(n)) : n\in\omega\}$ and $\max[X]=\{\max(X(n)) : n\in\omega\}$. We will be dealing with a particular type of a Matet-adequate family, the ones defined by Mildenberger in \cite{mildenberger2024exactly} with respect to two non-isomorphic Ramsey ultrafilters.

\begin{defi}[Definition 5.1, Mildenberger \cite{mildenberger2024exactly}]

Let $\mathcal{U}_0$ and $\mathcal{U}_1$ be two non-isomorphic Ramsey ultrafilters. Let $\bar{\mathcal{U}}=\langle \mathcal{U}_{0}, \mathcal{U}_{1}\rangle$. We define $\mathrm{FIN}^{[\infty]}(\bar{\mathcal{U}})=\{X \in \mathrm{FIN}^{[\infty]} : (\forall A\in \mathcal{U}_{0})(\forall B\in\mathcal{U}_{1}) (\exists Y \leq X) (\forall n) (\min(Y(n))\in A, \max(Y(n))\in B)\}$.

\end{defi}
In particular, if $X\in \mathrm{FIN}^{[\infty]}(\bar{\mathcal{U}})$, then $\min[X]\in \mathcal{U}_0, \max[X]\in \mathcal{U}_1$.

\begin{thm}[Theorem~5.2, Mildenberger \cite{mildenberger2024exactly}]

Given two non-isomorphic Ramsey ultrafilters $\bar{\mathcal{U}}=\langle \mathcal{U}_0, \mathcal{U}_1\rangle$, $\mathrm{FIN}^{[\infty]}(\bar{\mathcal{U}})$ is Matet-adequate.

\end{thm}

By a standard density argument, if $\eta$ is a generic block sequence for the forcing $\mathbb{MM}(\mathrm{FIN}^{[\infty]}(\bar{\mathcal{U}}))$, then $\min[\eta]$ pseudo-intersects $\mathcal{U}_0$, and $\max[\eta]$ pseudo-intersects $\mathcal{U}_1$. 

As in the proof of the main theorem in the case $n=2$, we start with a pair of non-isomorphic Ramsey ultrafilters $\bar{\mathcal{U}}^0=\langle \mathcal{U}^0_0, \mathcal{U}^0_1 \rangle$, whose existence is guaranteed by the \textsf{CH}. We define the countable support iteration $\mathbb{P}_{\omega_2}=\langle \mathbb{P}_{\alpha}, \undertilde{\mathbb{MM}}(\mathrm{FIN}^{[\infty]}(\bar{\mathcal{U}}^{\alpha})) : \alpha<\omega_2\rangle$, where for each $\xi < \omega_2: \mathbb{P}_{\xi}\Vdash \undertilde{\mathcal{U}}^{\xi}_i\supseteq \bigcup_{\iota<\xi} \undertilde{\mathcal{U}}^{\iota}_i$, for $i\in\{0,1\}$. Again, the existence of such Ramsey ultrafilters is guaranteed by the \textsf{CH} at each intermediate stage. We show that for any generic $G\subseteq \mathbb{P}_{\omega_2}$, $$\mathcal{U}^{\omega_2}_0=\bigcup_{\xi<\omega_2} \undertilde{\mathcal{U}}_0^{\xi}[G|_{\xi}] \text{ and } \mathcal{U}^{\omega_2}_1=\bigcup_{\xi<\omega_2}\undertilde{\mathcal{U}}_1^{\xi}[G|_{\alpha}]$$ are, up to isomorphism, the only $Q$-points in $\mathbf{V}[G]$. As in the previous section, it suffices to prove the following proposition:

\begin{prp}\label{prop:destruction2}
Let $E$ be a $Q$-point and $\bar{\mathcal{U}}=\langle \mathcal{U}_0, \mathcal{U}_1\rangle$ a pair of non-isomorphic Ramsey ultrafilters such that $E$ is not isomorphic to $\mathcal{U}_0$ or $\mathcal{U}_1$. Write $\mathcal{H}:=\mathrm{FIN}^{[\infty]}(\bar{\mathcal{U}})$. Let $\undertilde{R}$ be a $\mathbb{MM}(\mathcal{H})$-name such that $\mathbb{MM}(\mathcal{H}) \Vdash ``\undertilde{R}$ has the $h$-Laver property". Then $\mathbb{MM}(\mathcal{H}) \ast \undertilde{R}\Vdash ``E$ cannot be extended to a $Q$-point".
\end{prp}

Proposition~\ref{prop:destruction2} directly follows from the following three lemmas. The reader can check the proof of Proposition~\ref{prop:destruction} to see why these lemmas are enough to derive Proposition~\ref{prop:destruction2}.

\begin{lem}\label{lemma-deciding2}

Let $\undertilde{C}$ be a $\mathbb{MM}(\mathcal{H})$-name for a subset of $\omega$, and let $\langle a, X\rangle \in \mathbb{MM}(\mathcal{H})$ be an arbitrary condition. Then there is $X^*\in\mathcal{H}|X$, and for all $b\in[X^*]^{[<\infty]}$, some $C_b\in [\omega]^{<\omega}$ such that $\langle a^{\smallfrown}b, X^*/b\rangle \Vdash_{\mathbb{MM}(\mathcal{H})} \undertilde{C}\cap\max(b(|b|-1))^+=C_b$. Moreover, the following also hold:
\begin{enumerate}[label=(\roman*)]
    \item For all $b\in[X^*]^{[<\infty]}$, $s\in[X^*/b]$, and $t_0,t_1\in [X^*/s]$, we have $C_{b^{\smallfrown}t_0}\cap \max(s)^+=C_{b^{\smallfrown}t_1}\cap \max(s)^+$. \label{lemma-deciding2a}
    \item For all $b\in[X^*]^{[<\infty]}$, $s\in[X^*/b]$, and $t_0,t_1\in [X^*/s]$, we have $C_{b^{+\smallfrown}t_0}\cap \max(s)^+=C_{b^{+\smallfrown}t_1}\cap \max(s)^+$ (here, $b^{+\smallfrown}t_i:=\langle b(0),\ldots, b(|b|-1)\cup t_i\rangle$). \label{lemma-deciding2b}
\end{enumerate}


\end{lem}

\begin{proof}
Suppose $\undertilde{C}$ and $\langle a, X\rangle$ are as in the statement of the lemma. The proof proceeds similarly to the proof of Lemma~\ref{lemma:Laver property}. By pure decision, we find $C_{\varnothing}\subseteq \max(a(|a|-1))$ and some $X_0\leq X$ such that $\langle a, X_0\rangle \Vdash_{\mathbb{MM}(\mathcal{H})} \undertilde{C}\cap  \max(a(|a|-1))=C_{\varnothing}$.

Assuming that $X_0\geq X_1\geq\ldots\geq X_{n-1}$ have been constructed, first, define $Y_0=X_{n-1}/X(n-1)$. Let $\langle b_i : i\in m\rangle$ list $\{b\in[X\upharpoonright n]^{[<\infty]} : d_X(b)=n-1\}$, for some $m\in\omega$. Given $Y_i$ for $i\in m$, by pure decision, define $C_{b_i}$ and $Y_{i+1}\in \mathcal{H}|Y_i$ so that $\langle a^{\smallfrown} b_i, Y_{i+1}\rangle \Vdash_{\mathbb{MM}(\mathcal{H})} \undertilde{C}\cap \max(b_i(|b_i|-1))= C_{b_i}$. Finally, define $X_n=Y_m$ and proceed the induction. By Lemma~\ref{Matetadequateproperties}, we find some diagonal intersection $X'\in \mathcal{H}|X_0$ of the sequence $\langle X_0, X_1,\ldots\rangle$. Then $\langle a, X'\rangle$ satisfies the first part of the statement.

In order to satisfy \ref{lemma-deciding2a} and \ref{lemma-deciding2b}, let $X'_0=X'$ and assume that $X'_0\geq X'_1\geq\ldots\geq X'_{n-1}$ have been constructed. Let $\langle b_i : i\in m'\rangle$ list $[X'\upharpoonright n]^{[<\infty]}$, for some $m'\in\omega$. Set $Y'_0=X'_{n-1}/X'(n-1)$, and given $Y'_i$ for $i\in m'$, define a coloring $c_{b_i}: [Y'_i/X'(n)]\to \ran(c_{b_i})$ by $c_{b_i}(t)=C_{b_i^{\smallfrown}t}\cap \max(X'(n))^+$, for $t \in [Y'_i/X'(n)]$. Note that $\ran(c_{b_i})$ is finite, so by definition of a Matet-adequate family, we can find $Y'_{i+1}\in\mathcal{H}|Y'_i$ such that $c_{b_i}\upharpoonright [Y'_{i+1}]$ is constant. At the end, we let $X'_n=Y'_{m'}$, and find $X''\in\mathcal{H}|X'_0$ which is a diagonal intersection for the sequence $\langle X'_0, X'_1,\ldots\rangle$. To finish, we do the same construction as above to $X''$, the only difference being in the definitions of the colorings $c_{b_i}$; for which we use $C_{b_i^{+\smallfrown}t}$ instead of $C_{b_i^{\smallfrown}t}$, and define $X^*\in\mathcal{H}$ to be a diagonal intersection of the relevant sequence.
\end{proof}

\begin{lem}
Let $\undertilde{C}$ be a $\forcing$-name for a subset of $\omega$ and let $\langle a,  X\rangle \in \mathbb{MM}(\mathcal{H})$ be such that $\langle a, X\rangle \Vdash_{\mathbb{MM}(\mathcal{H})} \forall i \in \omega: |\undertilde{C} \cap (\max(\dot{\eta}(i-1)),\max{\dot{\eta}(i)}]|\leq h(i)$.
	
There exists an interval partition $\{[k_i, k_{i+1}): i \in \omega\}$ of $\omega$ and a strengthening of $\langle a, X\rangle$ which forces that
		  $$\forall i \in \omega\setminus \{0,1\} :(\undertilde{C}\setminus\max(a)^+)\cap [k_{i-1}, k_i)\neq \varnothing \Rightarrow \begin{cases}
        \left(\bigcup \dot{\eta}\right)\cap [k_{i-2}, k_{i-1})\neq \varnothing, & \text{or} \\ \left(\bigcup \dot{\eta}\right)\cap [k_{i-1}, k_i)\neq \varnothing, & \text{or} \\ \left(\bigcup \dot{\eta}\right)\cap [k_i, k_{i+1})\neq \varnothing.
    \end{cases}$$
\end{lem}

\begin{proof}
The proof is analogous to the proof of Lemma~\ref{lem:close}: Without loss of generality, we assume that the minimum of the $j$'th element of $a^{\smallfrown}X$ is greater than $h(|a|+j+1)$, let $X^*\in\mathcal{H}|X$ be as given by Lemma~\ref{lemma-deciding2}, and define the interval partition via finite approximations of a certain countable elementary submodel of $\mathcal{H}_{\chi}$. The only case that requires some additional care is proving that if $\langle a^{\smallfrown}b, X^{\ast}/b\rangle \Vdash_{\mathbb{MM}(\mathcal{H})} n \in \undertilde{C}$, where $\min(b(|b|-1)) < n < \max(b(|b|-1))$, then some element of $b(|b|-1)$ lies in the interval immediately succeeding the interval containing $n$. However, this case is easily handled by item (ii) in the conclusion of Lemma~\ref{lemma-deciding2}.
\end{proof}

The final ingredient is the following slightly modified version of Lemma~5.10 of \cite{mildenberger2024exactly}.

\begin{lem}\label{lemma-final}

Suppose $\{[k_i, k_{i+1})\}_{i\in\omega}$ is an increasing interval partition of $\omega$, and $\mathcal{W}\ncong\mathcal{U}_0,\mathcal{U}_1$ is a $Q$-point. Then for all $X\in\mathcal{H}$, there are $Y\in\mathcal{H}|X$ and $W\in\mathcal{W}$ such that $\bigcup Y\cap [k_{i-1}, k_i)\neq\varnothing\Rightarrow \begin{cases}
    W\cap [k_{i-2}, k_{i-1})=\varnothing, & \text{and} \\ W\cap[k_{i-1}, k_i)=\varnothing, & \text{and} \\ W\cap [k_i, k_{i+1})=\varnothing.

\end{cases}$

\end{lem}

\begin{proof}

First, by Lemma~\ref{Matetadequateproperties}, find some $X'\in\mathcal{H}|X$ such that every $[\min X'(n), \max X'(n)]$ contains some interval $[k_{i-1}, k_i)$, and such that no interval is intersected by $X'(n)$ and $X'(m)$ for distinct $n$ and $m$. 
Next, define $l_0=0$, and $l_i=\max(X'(i))$ for all $0<i<\omega$. Define $h_0:\omega\to\omega$ by $h_0\upharpoonright [l_{2n}, l_{2n+2})=n$ for all $n \in \omega$. Define $h_1:\omega\to\omega$ by $h_1\upharpoonright [l_0, l_1)=0$ and $h_1\upharpoonright [l_{2n+1}, l_{2n+3})=n$ for all $n\in\omega$.

Find $W\in\mathcal{W}, A\in\mathcal{U}_0, B\in \mathcal{U}_1$ such that
\begin{enumerate}[label=(\roman*)]
    \item $A\subseteq \min[X'], B\subseteq \max[X']$,
    \item $h_0[W]\cap (h_0[A] \cup h_0[B])=\varnothing$,
    \item $h_1[W]\cap (h_1[A] \cup h_1[B])=\varnothing$.
\end{enumerate}
Let $\{n_i\}_{i\in\omega}$ increasingly enumerate those $n$ such that $\min(X'(n))\in A$ or $\max(X'(n))\in B$. Let $Y(i)=X'(n_i)$ for all $i\in\omega$. Then $Y\in\mathcal{H}$, and it is straightforward to check that $Y$ and $W$ are as desired.
\end{proof}

\begin{rmk}
Note that dominating reals are added at each stage of the iteration in the preceding proof, as well as in the iteration considered in the proof of the main theorem. Hence, both models satisfy $\mathfrak{b}=\mathfrak{d}=\mathfrak{c}=\omega_2$. Since $\mathfrak{b} \leq \mathfrak{u}$ is provable in \textsf{ZFC} (see Solomon~\cite{solomon1977families}), it follows that $\mathfrak{u}=\mathfrak{d}$ in both models, and hence they contain $2^{\mathfrak{c}}$-many near-coherence classes of ultrafilters by Banakh and Blass~\cite{banakh2006number}.
\end{rmk}

\begin{rmk}
Looking at the model for a unique $Q$-point in \cite{halbeisen2025there}, or at the model for exactly $n$ $Q$-points in the previous section, one might have the idea that after iteratively pseudo-intersecting a \emph{stable ordered-union ultrafilter} on $\mathrm{FIN}$ (see \cite{Blass} for the definition of such ultrafilters), one will obtain a model with exactly two $Q$-points: $\mathcal{U}^{\omega_2}_{\min}$ and $\mathcal{U}^{\omega_2}_{\max}$ -- the $\min$ and $\max$ projections of the final stable ordered union ultrafilter $\mathcal{U}^{\omega_2}$. This idea faces the following obstacle in practice: At stage $\xi+1$ of the iteration, there is a stable ordered union ultrafilter $\mathcal{U}^{\xi+1}$ extending $\mathcal{U}^{\xi}$, as well as $Q$-points $\mathcal{W}_\alpha$ for $\alpha \in 2^{\mathfrak{c}}$, such that each $\mathcal{W}_\alpha$ is not isomorphic to $\mathcal{U}^{\xi+1}_{\min}$ or $\mathcal{U}^{\xi+1}_{\max}$, but such that $\operatorname{core}(\mathcal{U}^{\xi+1})\subseteq \mathcal{W}_\alpha$ (see Eisworth~\cite{Eisworth} for the definition of the core). This follows by slightly adapting a construction due to Mildenberger~\cite[Theorem~2.1]{mildenberger2011milliken}. After forcing with $\mathbb{MM}(\mathcal{U}^{\xi+1})$, these $\mathcal{W}_{\alpha}$ may still be extendable to $Q$-points.

In this section, we instead used $\mathbb{MM}(\mathrm{FIN}^{[\infty]}(\bar{\mathcal{U}}))$. As is the case for Matet-forcing (see \cite{mildenberger2024exactly}), this forcing factors as a two-step iteration $\mathrm{FIN}^{[\infty]}(\bar{\mathcal{U}})*\mathbb{MM}(\undertilde{\mathcal{U}})$, where $\mathrm{FIN}^{[\infty]}(\bar{\mathcal{U}})$ adds a generic stable ordered-union $\undertilde{\mathcal{U}}$ with $\undertilde{\mathcal{U}}_{\min}=\mathcal{U}_0$ and $\undertilde{\mathcal{U}}_{\max}=\mathcal{U}_1$, and the second step of the iteration pseudo-intersects this stable ordered-union. In some sense, this is the most natural fix for the obstacle we described in the previous paragraph: For any $Q$-point $\mathcal{W}$ in the ground model with $\mathcal{W}\ncong \mathcal{U}_0, \mathcal{U}_1$, by Lemma~\ref{lemma-final}, the generic $\undertilde{\mathcal{U}}$ will satisfy $f(\operatorname{core}(\mathcal{\undertilde{U}}))\nsubseteq f(\mathcal{W})$ for any finite-to-one $f \in \func$; hence $\mathcal{W}$ will not be extendable to a $Q$-point after pseudo-intersecting the generic stable ordered-union $\undertilde{\mathcal{U}}$, by the same method as used in Section 3. 
\end{rmk}

\section{Open Problems}

\begin{problem}
Is it consistent that there are exactly $\aleph_0$-many $Q$-points?
\end{problem}
Note that a positive answer would require all but finitely many of these $Q$-points to be non-Ramsey, since for any countable set $\{\mathcal{U}_k : k\in \omega\}$ of pairwise non-isomorphic Ramsey ultrafilters, the ultrafilter 
\[\mathcal{U}_0\text{-}\sum_{k > 0}\mathcal{U}_k:= \{X \subseteq \omega \times \omega: \{k > 0:\{i \in \omega: \langle k,i\rangle \in X\}\in \mathcal{U}_k\}\in \mathcal{U}_0\}\]
on $\omega \times \omega$ is a non-Ramsey $Q$-point (see \cite[Section 2.1]{blass2015next}). Hence, by partitioning $\{\mathcal{U}_k : k>0\}$ into an almost disjoint family of size $\mathfrak{c}$, each $\mathcal{U}_0$-indexed sum over one of the $\mathfrak{c}$-many pieces yields a $Q$-point, and these will be pairwise non-isomorphic by an old result due to Rudin (\cite{rudin1965types}). This suggests the following second question:

\begin{problem}
Is it consistent that there is a unique $Q$-point, and it is non-Ramsey?
\end{problem}

\bibliographystyle{plain}
\bibliography{n-q-points}

\newcommand{\noop}[1]{}
\begin{thebibliography}{10}

\bibitem{abraham2009proper}
Uri Abraham.
\newblock Proper forcing.
\newblock In {\em Handbook of set theory}, pages 333--394. Springer, 2009.

\bibitem{banakh2006number}
Taras Banakh and Andreas Blass.
\newblock The number of near-coherence classes of ultrafilters is either finite or {$2^{\mathfrak c}$}.
\newblock In {\em Set Theory: Centre de Recerca Matem{\`a}tica Barcelona, 2003--2004}, pages 257--273. Springer, 2006.

\bibitem{blass86nearcoherence}
Andreas Blass.
\newblock Near coherence of filters. {I}. {C}ofinal equivalence of models of arithmetic.
\newblock {\em Notre Dame Journal of Formal Logic}, 27(4):579--591, 1986.

\bibitem{Blass}
Andreas Blass.
\newblock Ultrafilters related to {H}indman’s finite-unions theorem and its extensions.
\newblock {\em Contemporary Mathematics}, 65(3):89--124, 1987.

\bibitem{blass2009combinatorial}
Andreas Blass.
\newblock Combinatorial cardinal characteristics of the continuum.
\newblock In {\em Handbook of set theory}, pages 395--489. Springer, 2009.

\bibitem{blass2015next}
Andreas Blass, Natasha Dobrinen, and Dilip Raghavan.
\newblock The next best thing to a {$P$}-point.
\newblock {\em The Journal of Symbolic Logic}, 80(3):866--900, 2015.

\bibitem{blass1989near}
Andreas Blass and Saharon Shelah.
\newblock Near coherence of filters. {III}. {A} simplified consistency proof.
\newblock {\em Notre Dame Journal of Formal Logic}, 30(4):530--538, 1989.

\bibitem{BrendleGarciaAvila}
J\"org Brendle and Luz~Mar\'ia Garc\'ia~\'Avila.
\newblock Forcing-theoretic aspects of {H}indman's theorem.
\newblock {\em Journal of the Mathematical Society of Japan}, 69(3):1247--1280, 2017.

\bibitem{CalderonDiPriscoMijares}
Daniel Calder\'on, Carlos~Augusto Di~Prisco, and Jos\'e{}~G. Mijares.
\newblock Ramsey subsets of the space of infinite block sequences of vectors.
\newblock {\em Fundamenta Mathematicae}, 257(2):189--216, 2022.

\bibitem{Eisworth}
Todd Eisworth.
\newblock Forcing and stable ordered-union ultrafilters.
\newblock {\em The Journal of Symbolic Logic}, 67(1):449--464, 2002.

\bibitem{cov(M)<c}
David~Jos\'e Fern\'andez-Bret\'on.
\newblock Stable ordered union ultrafilters and {${\rm cov}(\mathcal M)<\mathfrak{c}$}.
\newblock {\em The Journal of Symbolic Logic}, 84(3):1176--1193, 2019.

\bibitem{PFIN}
Luz~Mar\'ia Garc\'ia-\'Avila.
\newblock A forcing notion related to {H}indman's theorem.
\newblock {\em Archive for Mathematical Logic}, 54(1-2):133--159, 2015.

\bibitem{halbeisen2025there}
Lorenz Halbeisen, Silvan Horvath, and Saharon Shelah.
\newblock A unique ${Q}$-point and infinitely many near-coherence classes of ultrafilters.
\newblock arXiv:2505.17960, (2025).

\bibitem{hind}
Neil Hindman.
\newblock Finite sums from sequences within cells of a partition of {$N$}.
\newblock {\em Journal of Combinatorial Theory. Series A}, 17:1--11, 1974.

\bibitem{mildenberger2011milliken}
Heike Mildenberger.
\newblock On {M}illiken-{T}aylor ultrafilters.
\newblock {\em Notre Dame Journal of Formal Logic}, 52(4):381--394, 2011.

\bibitem{mildenberger2024exactly}
Heike Mildenberger.
\newblock Exactly two and exactly three near-coherence classes.
\newblock {\em Journal of Mathematical Logic}, 24(01):2350003, 2024.

\bibitem{millan2007note}
Andres Mill{\'a}n.
\newblock A note about special ultrafilters on $\omega$.
\newblock In {\em Topology Proc}, volume~31, pages 219--226, 2007.

\bibitem{miller1980there}
Arnold~W Miller.
\newblock There are no {Q}-points in {L}aver's model for the {B}orel conjecture.
\newblock {\em Proceedings of the American Mathematical Society}, pages 103--106, 1980.

\bibitem{milliken}
Keith~R. Milliken.
\newblock Ramsey's theorem with sums or unions.
\newblock {\em Journal of Combinatorial Theory. Series A}, 18:276--290, 1975.

\bibitem{rudin1965types}
Mary~Ellen Rudin.
\newblock Types of ultrafilters.
\newblock In {\em Topology Seminar Wisconsin}, pages 147--151, 1965.

\bibitem{shelah1998distributivity}
Saharon Shelah and Otmar Spinas.
\newblock The distributivity numbers of finite products of {${\mathscr P}(\omega)/\mathrm{fin}$}.
\newblock {\em Fundamenta Mathematicae}, 158(1):81--93, 1998.

\bibitem{solomon1977families}
RC~Solomon.
\newblock Families of sets and functions.
\newblock {\em Czechoslovak Mathematical Journal}, 27(4):556--559, 1977.

\bibitem{Tay}
Alan~D. Taylor.
\newblock A canonical partition relation for finite subsets of {$\omega$}.
\newblock {\em Journal of Combinatorial Theory. Series A}, 21(2):137--146, 1976.

\end{thebibliography}

\end{document}